\documentclass{amsart}
\usepackage[T1]{fontenc}
\usepackage[utf8]{inputenc}

\usepackage{hyperref}

\theoremstyle{definition}
\newtheorem{definition}{Definition}[section]

\theoremstyle{plain}
\newtheorem{theorem}[definition]{Theorem}
\newtheorem{lemma}[definition]{Lemma}
\newtheorem{proposition}[definition]{Proposition}

\DeclareMathOperator{\rk}{rk}

\DeclareMathOperator{\Aut}{Aut}

\DeclareMathOperator{\End}{End}

\DeclareMathOperator{\Bir}{Bir}

\DeclareMathOperator{\disc}{discr}
\DeclareMathOperator{\Num}{Num}

\newcommand{\QQ}{\mathbb{Q}}
\newcommand{\FF}{\mathbb{F}}
\newcommand{\RR}{\mathbb{R}}
\newcommand{\CC}{\mathbb{C}}
\newcommand{\ZZ}{\mathbb{Z}}
\newcommand{\NN}{\mathbb{N}}

\renewcommand{\O}{\mathcal{O}}

\title[K3 Salem]{On the stable dynamical spectrum of complex surfaces}

\author{Simon Brandhorst}
\address{Insitut für Algebraische Geometrie, Leibniz Universität Hannover,
	Welfengarten 1, 30167 Hannover, Germany}
\email{brandhorst@math.uni-hannover.de}
\date{August, 2017}

\begin{document}
\begin{abstract}
We characterize Salem numbers which have some power arising as dynamical degree of an automorphism on a complex (projective) 2-Torus, K3 or Enriques surface.
\end{abstract}
\maketitle
\section{Introduction}
To a bimeromorphic transformation $F: X \dashrightarrow X$ of a Kähler surface one can associate its \emph{dynamical degree}
\[\lambda(F)=\limsup_{n \rightarrow \infty} (||(F^n)^*||)^{1/n},\] 
where $F^*$ denotes the action on $H^2(X,\ZZ)$ and $||\cdot||$ is any norm on $\End(H^2(X,\ZZ))$. The dynamical degree is a bimeromorphic invariant of $(X,F)$ which measures the dynamical complexity of $F$. In the projective case, it describes the asymptotic degree growth of defining equations for $F$. We call $F$ algebraically stable if $(F^n)^*=(F^*)^n$.
Passing to a bimeromorphic model of $(X,F)$ this can be achieved for surfaces \cite{diller_favre:dynamics_surface}. Hence, in this case the dynamical degree is an algebraic integer - the spectral radius of $F^*$.
Its logarithm $\log \lambda(F)$ is an upper bound for the topological entropy $h(F)$. If moreover $F: X \rightarrow X$ is an automorphism, then they agree and $\lambda(F)$ is a Salem number, that is, an algebraic integer $\lambda>1$ which is Galois conjugate to $1/\lambda$ and all whose other conjugates lie on the unit circle.
Conversely, if $\lambda(F)$ is a Salem number of degree at least $4$, then $F$ is conjugate to an automorphism on some birational model of $X$ \cite[Thm A]{cantat:dynamical_degrees}.
If the dynamical degree of an automorphism $F$ of a surface $X$ is $\lambda(F)>1$, then $X$ is a blow up of the projective plane in at least $10$ points, or a blow up of a 2-Torus, a K3 or an Enriques surface \cite{cantat:classification}.
Let $\mathcal{K}$ be a class of surfaces. We define the dynamical spectrum of surfaces of type $\mathcal{K}$ as
\[\Lambda(\mathcal{K},\CC)= \bigcup \{\lambda(F) | F \in \Bir(X),\; X/\CC \mbox{ is a } \mathcal{K} \mbox{ surface}\},\]
and its counterpart for projective surfaces
\[\Lambda^{proj}(\mathcal{K},\CC)= \bigcup \{\lambda(F) | F \in \Bir(X),\; X/\CC \mbox{ is a projective } \mathcal{K} \mbox{ surface}\}.\]
If $X$ is a surface of type $\mathcal{K} \in \{$ 2-Torus, K3, Enriques $\}$, then its canonical divisor is nef and hence $\Bir(X)=\Aut(X)$.
The Kummer construction and the fact that the universal cover of an Enriques surface is a K3 surface give the inclusions 
\[\Lambda(\mbox{2-Tori},\CC), \Lambda(\mbox{Enriques},\CC)\subseteq \Lambda(K3,\CC).\]
For rational surfaces the contribution to the dynamical spectrum coming from automorphisms is described in terms of Weyl groups in \cite{uehara:rational_automorphism_entropy}. However, for a concrete Salem number it seems to be hard to decide whether it is a spectral radius of an element of a certain Weyl group. 
For genuinely birational maps on rational surfaces, dynamical degrees may also be Pisot numbers and less seems to be known \cite{cantat:dynamical_degrees}.
The dynamical spectrum of complex 2-Tori (respectively Abelian surfaces) is completely described in \cite{reschke:tori_salem,reschke:abelian_salem}.
In a very recent preprint \cite[Cor. 1.1]{lenny:equivariant_witt}
Bayer-Fluckiger and Taelman completely characterize which Salem numbers of degree $22$ arize as dynamical degrees of (non-projective) K3 surfaces.

However, in lower degrees or projective K3 surfaces as well as Enriques surfaces our picture is much less complete. 
In each even degree $d$ there is a minimal Salem number $\lambda_d$. Conjecturally, the smallest Salem number is Lehmer's number $\lambda_{10}\approx 1.17628$. In \cite{mcmullen:minimum} McMullen gives a strategy to decide whether a single given Salem number $\lambda$ is the dynamical degree of an automorphism of a complex projective K3 surface. 
This strategy is then applied in \cite{mcmullen:minimum,brandhorst_alonso:minimal_salem} to show that the minimal Salem numbers
\[\lambda_d \in \Lambda^{proj}( K3,\CC) \quad \Leftrightarrow \quad 14,16\neq d \leq 18.\]
 Using the strategy in \cite{mcmullen:minimum} and the improved positivity test from \cite{brandhorst_alonso:minimal_salem} one easily obtains
 \[\lambda_{14}^9,\lambda_{16}^7,\lambda_{20}^{11} \in \Lambda^{proj}( K3,\CC).\]
 In the non-projective realm we get
 \[\lambda_{14},\lambda_{16},\lambda_{20},\lambda_{22} \in \Lambda( K3,\CC).\]

  Given the dynamical context of the question, it is natural to ask for stable realizations of dynamical degrees instead, that is, whether there exists some power of a given Salem number which arises as a dynamical degree.
 \begin{theorem}\label{thm:main2}
 	Let $\lambda$ be a Salem number of degree $d\leq 20$. Then there is an $n \in \NN$, a projective K3 surface $X$ and an automorphism $F: X \rightarrow X$ with dynamical degree $\lambda(F)=\lambda^n$.
 \end{theorem}

Considering the same question for complex tori and Enriques surfaces one finds that the answer depends only on the Betti and Hodge numbers. We can derive Theorem \ref{thm:main2} from the more general Theorem \ref{thm:main} below.

\begin{theorem}\label{thm:main}
	Let $\lambda$ be a Salem number with minimal polynomial $s(x)\in \ZZ[x]$ of degree $d$. Fix a class of surfaces $\mathcal{K} \in \{$2-Torus, K3, Enriques$\}$ and denote by $b_2(\mathcal{K})\in\{6,22,10\}$ the second Betti number of a surface in this class.
	Then there exists an $n \in \NN$ with $\lambda^n \in \Lambda(\mathcal{K},\CC)$ if and only if 
	\begin{enumerate}
		\item $d< b_2(\mathcal{K})$ or
		\item $d=b_2(\mathcal{K})$ and $-s(1)s(-1) \in \left(\QQ^\times\right)^2$.
	\end{enumerate}
If additionally $d\leq h^{1,1}(\mathcal{K})$, then we can find $n'\in \NN$ with $\lambda^{n'} \in \Lambda^{proj}(\mathcal{K},\CC)$.
\end{theorem}

The proof proceeds as follows. All Tori (resp. K3/Enriques) are diffeomorphic. Hence, the isometry class of the lattice $H^2(X,\ZZ)$ is independent of which Torus (resp. K3/Enriques) $X$ we have chosen. It is abstractly isomorphic to some fixed lattice $L$. Given an isometry $f \in O(L)$ it is possible, using some Torelli theorem, to decide whether $f$ is in the image of the natural representation $\Aut(X) \rightarrow O(H^2(X,\ZZ))\cong O(L)$. For this to be the case $f$ has to preserve some extra linear data such as a Hodge structure or has trivial mod $2$ reduction. The most intricate case is that of \emph{projective} K3 surfaces. There $f$ has to preserve a chamber of the positive cone - corresponding to the ample cone. Since it is usually infinite sided, it is notoriously difficult to control.
For a given concrete $f$ it is now algorithmically possible to decide whether this cone is preserved \cite{mcmullen:minimum,brandhorst_alonso:minimal_salem}.
However, the algorithm can only deal with a single isometry at a time. 
In Proposition \ref{prop:preserves_chamber}, we give a sufficient condition for a chamber to be preserved. We expect that it will be useful to study the stable dynamical spectrum of supersingular K3 surfaces (as in \cite{brandhorst_alonso:minimal_salem}) and IHSM manifolds (as in \cite{amerik:automorphisms_ihsm}) as well.

\subsection*{Acknowledgements}
I thank Víctor Gonzalez-Alonso, Curtis T. McMullen and Matthias Sch\"utt for comments and discussions on an early version of this paper.
The financial support of
the research training group GRK 1463 ''Analysis, Geometry and String Theory'' is gratefully acknowledged.
\section{Preliminaries}
In this section we review the necessary material from \cite{nikulin:quadratic_forms,mcmullen:minimum} concerning the theory of lattices, their isometries, discriminant forms, gluings and twists. 

\subsection{Lattices}
A \emph{lattice} is a finitely generated free abelian group $L$ equipped with a non-degenerate integer valued bilinear form
\[\langle \cdot , \cdot \rangle \colon L \times L \rightarrow \ZZ.\]
It is called \emph{even} if $\langle x , x \rangle \in 2\ZZ$ for all $x \in L$. For brevity we sometimes write $x.y$ for $\langle x,y\rangle$ and $x^2$ for $x.x$ where $x,y \in L$. 
The \emph{dual lattice} $L^\vee$ of $L$ is given by 
\[L^\vee = \{x \in L \otimes \QQ \mid \langle x , L \rangle \subseteq \ZZ\}.\]
Let $(e_i)$ be any $\ZZ$-basis of $L$, then the \emph{determinant} of $L$ is defined as the determinant of the Gram matrix $( e_i. e_j)_{ij}$. We call $L$ \emph{unimodular} if it is of determinant $\pm 1$.
An \emph{isometry} $M\rightarrow L$ of lattices is a homomorphism of $\ZZ$-modules preserving the bilinear forms. The orthogonal group $O(L)$ consists of the self isometries of the lattice $L$.
The \emph{signature} (pair) of a lattice is denoted by $(s_+,s_-)$ where $s_+$ (respectively $s_-$) is the number of positive (respectively negative) eigenvalues of the Gram matrix. A lattice is called \emph{indefinite} if both $s_+$ and $s_-$ are non-zero and \emph{hyperbolic} if it is indefinite and $s_+=1$. We denote by $U$, resp. $E_8$, the even unimodular lattices of signature $(1,1)$, respectively $(0,8)$. Indefinite, even unimodular lattices are classified up to isometry by their signature pair. 

The \emph{discriminant group} $D_L=L^\vee/L$ has cardinality $|\det L |$. 
If $L$ is an even lattice, then its discriminant group carries the \emph{discriminant form}, given by
\[q_L \colon D_L \rightarrow \QQ/2\ZZ \quad  \langle x , x\rangle \mod 2\ZZ.\] 
If $M\subseteq L$ are lattices of the same rank, then we call $L$ an \emph{overlattice} of $M$. Even overlattices $L$ of a lattice $M$ correspond bijectively to isotropic subgroups $L/M=H\subseteq D_M$, i.e., with $q_M|H=0$. 
For a prime number $p$ we denote by $\ZZ_p$ the p-adic integers and by $\QQ_p$ the p-adic numbers.
The discriminant form (and group) has an orthogonal decomposition into its $p$-primary parts $(q_L)_p$
\[q_L = \bigoplus_p \left((q_L)_p \colon (D_L)_p \rightarrow \QQ_p/2\ZZ_p\right)\]
where $(q_L)_p$ is the discriminant form of $L\otimes \ZZ_p$ (defined analogously).

\subsection{Embeddings and gluing}
An embedding of lattices $M \rightarrow L$ is called \emph{primitive} if $L/M$ is torsion free. Let $M\rightarrow L$ be a primitive embedding into a unimodular lattice $L$ and $N=M^\perp$ the orthogonal complement. It is primitive as well. 
We get an isomorphism 
$ \phi \colon D_M \rightarrow D_N$, with $q_N(\phi(x))=-q_M(x)$, called glue map. Conversely, given such a glue map, its graph 
\[\Gamma =\{x + \phi(x) \in D_M \oplus D_N \mid x \in D_M\}\]
is isotropic with respect to the discriminant quadratic form $q_{M\oplus N}$. Hence, it defines a unimodular overlattice $L$ via $L/\left(M\oplus N \right)=\Gamma$. 
Let $f \in O(M)$, $g\in O(N)$ be isometries. 
Then $f\oplus g \in O(M\oplus N)$ extends to the overlattice $L$ if and only if $\phi \circ \bar f = \bar g \circ \phi$ where $\bar f \in O(q_M)$ and $\bar g \in O(q_N)$ are the induced actions. 

\subsection{Twists}
A pair $(L,f)$ of a lattice $L$ and an isometry $f$ of $L$ with characteristic polynomial $s(x)\in \mathbb{Z}[x]$ is called a $s(x)$-lattice. 
Given a $s(x)$-lattice $(L,f)$ and $a \in \mathbb{Z}[f+f^{-1}]$, we obtain a new symmetric bilinear form on $L$ by setting
\[\left\langle g_1 , g_2 \right \rangle_a=\left\langle ag_1 , g_2 \right \rangle.\]
The lattice $L$ equipped with this new product is called the \textit{twist} of $L$ by $a$ and is denoted by $(L(a),f)$. Note that the twist of an even lattice stays even.
Twisting may change the signature and determinant of a lattice. However, we will only twist by a square $t^2 \in \ZZ[f+f^{-1}]$. Then $L(t^2)$ is isomorphic, via $x\mapsto tx$, to the sublattice $tL$ of $L$. In particular, the signature of $L(t^2)$ and $L$ coincide.
Of particular interest is the case when the characteristic polynomial $s(x)$ of $f$ is irreducible. Then $K=\QQ[f]$ is a degree two extension of the field $k=\QQ[f+f^{-1}]$. 
\begin{lemma}\label{lem:twist_split}	
Let $(L,f)$ be a $s(x)$-lattice with $s(x)$ irreducible. 
Suppose that $t \in \ZZ[f+f^{-1}]$ is a prime of $k$ split in $K$ of norm $p$ not dividing $2 \cdot \det L \cdot \disc p(x)$. Then the twisted p(x)-lattice $L(t^n)$ has determinant $\pm p^{2n}$ and discriminant quadratic form isomorphic to
\[q_{L(t^n)}\cong \frac{1}{p^n} \left( \begin{matrix}
0 & 1\\
1 & 0
\end{matrix}
\right).\]
\end{lemma}
\begin{proof}
Since the statement is local in $p$, we can tensor with $\ZZ_p$.
The factorization of $t$ in $K$ corresponds to the factorization of $s(x)$ in $\ZZ_p[x]$. Since $t$ is split, the corresponding factorization is $t=(f-a)(f-a^{-1})u(f)$ for some $a \in \ZZ_p$ and $u[f]\in \left(\ZZ_p[f]\right)^\times$. 
Hence, we can split off the combined eigenspace $E=\ker (f+f^{-1}-a-a^{-1})$ of $f$ for $a$ and $a^{-1}$.
We get $L(t^n)=E(t^n) \oplus E^\perp(t^n)$.
Since $t|E^\perp$ is invertible, we get that $\det E^\perp (t^n)=\det E^\perp$ is unimodular. Hence, up to units $\det E=\det L(t^n) =\det t^{n}=N^K_\QQ(t^n)=p^{2n}$. 
Then, in an eigenbasis the Gram Matrix of $E(t^n)$ is given by 
$\left( \begin{matrix}
0 & p^n\\
p^n & 0
\end{matrix}
\right)$. The matrix representing the discriminant form is now obtained by inverting the Gram matrix. 
Since $\det u(f)$ is a unit, $E^\perp(t^n)$ is unimodular. 
\end{proof}
 
\subsection{Positivity}
Let $L$ be an even lattice. A \emph{root} of $L$ is $r \in L$ with $r^2 = -2$. We denote the set of roots by $\Delta_L$. If $L$ is hyperbolic, we set
\[V_L=\{x \in L \mid x^2 >0, r.x \neq 0 \; \forall r \in \Delta_L\}\]
which is an open set. If $L$ is negative definite, we define
\[V_L=\{x \in L \mid x^2 <0, r.x \neq 0 \; \forall r \in \Delta_L\}.\]
In both cases the connected components of $V_L$ are called the \emph{chambers} of $V_L$. An isometry $f\in O(L)$ is called \emph{positive} if it preserves a chamber.
We denote by $O^+(L)$ the subgroup stabilizing each connected component of the light cone $V^0=\{x \in L \otimes \RR \mid x\neq 0, \; x^2=0\}$. 
A dual perspective on positivity is that of obstructing roots.
An \emph{obstructive root} for $f$ is $r \in \Delta_L$ such that there is no $h \in L$ with $h^\perp$ negative definite and $ h . f^i(r) >0$ for all $i \in \ZZ$. 
We call $r$ a \emph{cyclic root} for $f$ if 
\[r + f(r) + f^2(r) + \dots + f^i(r)=0\]
for some $i>0$. Cyclic roots are obstructing. If $L$ is negative definite, then every obstructing root is cyclic.
We have the following 
\begin{theorem}\cite[2.1]{mcmullen:minimum}\label{thm:obstructing_roots}
A map $f\in O^+(L)$ is positive if and only if it has no obstructing roots. The set of obstructing roots, modulo the action of $f$, is finite.
\end{theorem}
Let $L$ be hyperbolic and $f\in O(L)$ with spectral radius a Salem number $\lambda$. Denote by $\gamma$ the real plane spanned by the eigenspaces for $\lambda$ and $\lambda^{-1}$. Then the obstructing roots for $f$ are the cyclic roots together with the roots $r$ such that $r^\perp \cap \gamma$ is positive definite. To see this, note that the closure of any $f$-invariant chamber contains the intersection of $\gamma$ with the positive cone.

\section{Surfaces and their automorphisms}
Let $X$ be a either a $2$-Torus, a K3 surface or an Enriques surface.
Its second singular cohomology group modulo torsion $H^2(X,\ZZ)/tors$ equipped with the cup product is a unimodular lattice isomorphic to 
\[H^2(X,\ZZ)/tors \cong \begin{cases}
3U & \mbox{for } $X$ \mbox{ a 2-Torus} \\
U \oplus E_8 & \mbox{for } $X$ \mbox{ an Enriques surface} \\
3U \oplus 2 E_8 &\mbox{for } $X$ \mbox{ a K3 surface.}  
\end{cases}
\]
It admits a Hodge decomposition
\[H^2(X,\mathbb{Z})\otimes \mathbb{C} \cong H^2(X,\mathbb{C})=H^{2,0}(X) \oplus H^{1,1}(X) \oplus H^{0,2}(X)\]
where $H^{i,j}(X)\cong H^j(X,\Omega_X^i)$, $H^{i,j}(X)=\overline{H^{j,i}(X)}$ and $H^{1,1}(X)=(H^{2,0}(X) \oplus H^{0,2}(X))^\perp$ is hyperbolic.
By Lefschetz' Theorem on $(1,1)$ classes we can recover the numerical divisor classes from the Hodge structure as 
\[\Num(X)= H^{1,1}(X)\cap H^2(X,\mathbb{Z})/tors.\]
We note that a (compact) Kähler surface $X$ is projective if and only if there is a divisor of positive square, i.e. $\Num(X)$ has signature $(1,\rk \Num(X)-1)$. 
The transcendental lattice is defined as the smallest primitive sublattice $T\subseteq H^2(X,\mathbb{Z})$ whose complexification contains $H^{2,0} \subseteq T\otimes \mathbb{C}$. Since in our case $h^{2,0}\in \{0,1\}$, the Hodge structure on $T$ is irreducible.

Let $F\in \Aut(X)$ be an automorphism with dynamical degree $\lambda>1$. We call the minimal polynomial $s(x)\in \ZZ[x]$ of $\lambda$ a \emph{Salem polynomial}. Then $F^*$ is semisimple with characteristic polynomial $s(x)c(x)$ where $c(x)$ is a product of cyclotomic polynomials \cite[Thm. 3.2]{mcmullen:siegel_disk}.
Set $S=\ker s(F^*)\subseteq H^2(X,\ZZ)/tors$.
By irreducibility of $T(X)$, the minimal polynomial of $F^*|T(X)$ must be irreducible in $\QQ[x]$ too. Hence, either  $T(X)= S$  or $S\subseteq \Num(X)$.
Since the real eigenvectors for $\lambda$ and $\lambda^{-1}$ span a hyperbolic plane, the signature of $S$ is either  $(1,\deg s(x)-1)$ if $S\subseteq \Num(X)$ or in case $S=T(X)$ it is $(3,\deg s(x)-3)$. In the first case $X$ is projective and in the second not.\\

In the following three Lemmas we collect criteria for an isometry of the cohomology lattice to come from an automorphism of a surfaces.

\begin{lemma}\cite{barth:automorphism_enriques}\label{lem:torelli_enriques}
The automorphism group of a very general Enriques surface $X$ is the $2$-congruence subgroup given by the kernel of 
\[O^+(H^2(X,\ZZ))\rightarrow O(H^2(X,\ZZ)\otimes \FF_2).\]
It is of finite index in the orthogonal group of $H^2(X,\ZZ)/tors\cong U\oplus E_8$.
\end{lemma}

\begin{lemma}\label{lem:torelli_tori}
Let $s(x)$ be a Salem polynomial of degree $d$ and
$f \in O(3U)$ an isometry with characteristic polynomial $s(x)(x-1)^{6-d}$ which acts trivially on $3U \otimes \FF_2$.
Then one can find a complex $2$-torus $T$ with $3U=H^2(T,\ZZ)$ and $F \in \Aut(T)$ such that $F^* =f$. The $2$-torus $T$ is projective if and only if $S=\ker s(f)$ has signature $(1,d-1)$.
\end{lemma}
\begin{proof}
Let $L$ be a free $\ZZ$ module of rank $4$. An orientation on $L$ is an isomorphism $\det\colon \bigwedge^4 L \xrightarrow{\sim} \ZZ$. It gives rise to an even unimodular lattice which we identify with $3U$.
We can choose an eigenvector $\eta \in 3U \otimes \CC$ of $f \otimes \CC$ such that $\eta^2=0$ and $\eta.\bar \eta>0$. Since $\CC \eta$ is isotropic, it is in the image of the Plücker embedding $Gr(2,L\otimes \CC) \rightarrow \bigwedge^2 L \cong 3U$. Hence, we get a complex $2$-plane $S$ with $\bigwedge^2S=\CC \eta$ and $S \oplus \bar S=L\otimes \CC$. This defines a weight one Hodge structure on $L$, i.e., a complex $2$-torus $T$.
We can view $f$ as an isometry of $H^2(T,\ZZ)=\bigwedge^2 L\cong 3U$ which,
by construction, preserves the Hodge structure on $H^2(T,\ZZ)$. Since $f\otimes \FF_2$ is the identity, we can apply, \cite[V (3.2)]{BHPV:compact_complex_surfaces} to get an automorphism $F$ of $T$ with $F^*=\pm f$. However, both $F^*$ and $f$ stabilize each connected component of the positive cone of $H^{1,1}(T)$. Hence, they are equal. 
\end{proof}

\begin{lemma}\label{lem:torelli_k3}
Let $f \in O(3U\oplus 2E_8)$ be an isometry with characteristic polynomial $s(x)(x-1)^{22-d}$.Then one can find a K3 surface $X$, $F \in \Aut(X)$ and an isometry $\phi:3U\oplus 2E_8 \rightarrow H^2(X,\ZZ)$ such that $F^* =\phi \circ f \circ \phi^{-1}$ if and only if
\begin{enumerate}
	\item $S=\ker s(f)$ has signature $(3,d-3)$ or
	\item $S$ has signature $(1,d-1)$ and $f|S$ is positive.	
\end{enumerate}
In case (2) $X$ is projective and in case (1) not.
\end{lemma}
\begin{proof}
The lemma follows once we check the conditions of \cite[6.1]{mcmullen:minimum}.
If the signature of $S$ is $(3,d-3)$, then we take as period an eigenvector $\eta\in S\otimes \CC$ of $f$ with $\eta.\bar{\eta}>0$.
Since $f$ is the identity on $S^\perp$ there are no cyclic roots, and $f|S^\perp$ is positive.  
If the signature of $S$ is $(1,d-1)$, then we take as period a very general line in $S^\perp \otimes \CC$.
\end{proof}

\section{Proof of Theorem \ref{thm:main}}
In order to prove the main Theorem \ref{thm:main}, we need to produce isometries of certain lattices with given spectral radius. In general this can be difficult. Hence, we simplify the problem by asking for \emph{rational} isometries first. Indeed, here the answer is known as is displayed by the following Lemma \ref{lem:iso_QQ}. We postpone its proof till the end of this paper.
\begin{lemma}\label{lem:iso_QQ}
 Let $L \in \{3U, U\oplus E_8, 3U\oplus 2E_8\}$ and $s(x)$ be a Salem polynomial of degree $d$.
 Then there exists a rational isometry $f\in O(L\otimes \QQ)$ with characteristic polynomial $\det (xId-f)=s(x)(x-1)^{\rk L-d}$
 if and only if either 
 \begin{enumerate}
 	\item $d\leq \rk L-2$ or
 	\item $d=\rk L$ and $-s(1)s(-1)$ is a square.
 \end{enumerate} 
 In case $(1)$ we can find $f$ such that $\ker s(f)$ is hyperbolic.
 If the signature of $L$ is $(3,\rk L-3)$, then we can find $f$ such that $\ker s(f)$ has signature $(3,d-3)$.
\end{lemma}
\noindent
Typically, a rational isometry $f\in O(L\otimes \QQ)$ does not preserve $L$.
Since we are only considering the \emph{stable} dynamical spectrum, we may replace $f$ by some power $f^n$.

\begin{lemma}\label{lem:iso_ZZ}
	Let $L$ be a lattice and $f\in O(L\otimes \QQ)$ a rational isometry with \[\det(xId-f) \in \ZZ[x].\] 
	Then one can find $n\in \NN$ such that $f^n\in O(L)$. 
\end{lemma}
\begin{proof}
Since the characteristic polynomial of $f$ is integral, the $\ZZ$-module $\ZZ[f]L$ is finitely generated and of the same rank as $L$. Consequently, for the index $k=[\ZZ[f]L:L]$ we get the chain of inclusions
\[k \ZZ[f]L \subseteq L \subseteq \ZZ[f]L.\]
Conclude by taking $n\in \NN$ such that $f^n$ acts as the identity on the finite quotient $\ZZ[f]L/k\ZZ[f]L$. 
\end{proof}
We now have all the ingredients for the
\begin{proof}[Proof of Theorem \ref{thm:main} for Tori and Enriques surfaces]
A combination of Lemmas \ref{lem:iso_QQ} and \ref{lem:iso_ZZ} provides us with isometries of $L \in \{3U, U\oplus E_8\}$ with spectral radius some power of the desired Salem number. After raising the isometries to some sufficiently divisible power, we can assume that they satisfy
the conditions of Lemmas \ref{lem:torelli_enriques} and \ref{lem:torelli_tori}.
\end{proof}
The same argument completes the proof of the main theorem for non-projective K3 surfaces. It remains for us to control the positivity of the isometries to prove the result for projective K3 surfaces as well.

\begin{proposition}\label{prop:preserves_chamber}
	Let $f \in O(N)$ be an isometry of a hyperbolic lattice $N$ with characteristic polynomial a Salem polynomial $s(x)$.
	If 
	\[|\det N|> 4 \disc s(x),\]
	then $f$ preserves a chamber of the positive cone.
\end{proposition}
\begin{proof}
Let $L$ and $f$ be as in the proposition and denote by
\[\pi: S\otimes \RR \rightarrow \ker(f+f^{-1}-\lambda-\lambda^{-1})\] 
the orthogonal projection where $\lambda>1$ is a root of $s(x)$.
Suppose that $f$ does not preserve a chamber. Then, by Theorem \ref{thm:obstructing_roots}, there is an obstructive root $r\in L$. This means that $r^2=-2$ and $r^\perp$ crosses the geodesic $\gamma$ of $f$, i.e. $\pi(r)^2<0$.
Since $s(x)$ is irreducible over $\QQ$, $\ZZ[f]r$ 
is a sublattice of full rank, and hence
\[|\det N|\leq |\det \ZZ[f]r|.\]
The basic idea at this point is that the obstructing roots modulo the action of $f$ lie in some compact fundamental domain in $N\otimes \RR$ depending only on $s(x)$. Then we can maximize $|\det \ZZ[f]r|$ over all $r$ in this fundamental domain. 
We extend the bilinear form to a $\CC$-linear form on $N\otimes \CC$ and compute the determinant in an eigenbasis of $f$. 
We can find $u_1,u_2 \in N \otimes \RR$ and $v_i \in N \otimes \CC$ such that 
\[f(u_1)=\lambda u_1, \;\; f(u_2)=1/\lambda u_2, \;\;f(v_i)=\alpha_iv_i, \;\;i\in \{1,\dots, k\} \]
where $\lambda>1, 1/\lambda$, $\alpha_i,\overline{\alpha}_i$ are the complex roots of $s(x)$ and $\deg s(x)=2k+2$.
After rescaling, we may assume that $\langle u_1,u_2 \rangle=1$ and $\langle v_i,\overline{v_i} \rangle=-1$ for $i \in \{1, \dots, k\}$. Now, write 
\[r=x_1u_1+x_2u_2+\sum_{i=1}^k (y_i v_i + \overline{y_iv_i})\] 
for $x_1,x_2 \in \RR,y_i \in \CC$. The Van-der-Monde determinant yields that
\[|\det \langle f^i(r),f^j(r)\rangle|= |x_1x_2|^2 \prod_{i=1}^k|y_i|^4 \disc s.\]
Since $r$ is obstructing, $x_1x_2=\langle x_1u_1,x_2u_2\rangle \in [-2,0)$ and 
in these coordinates $r^2=x_1x_2 - \sum_{i=1}^k |y_i|^2=-2$, i.e. the coordinates of $r$ lie in the set
\[K=\left\{(x,y) \in \RR^{2}\times \CC^{k} : x_1x_2 \in [-2,0), \; x_1x_2 - \sum_{i=1}^k |y_i|^2=-2 \right\}.\]
Then, assuming $k\neq 0$,
\begin{eqnarray}
|\det N|							&\leq &|\det \langle f^i(r),f^j(r)\rangle|\\
 &\leq& \sup \left\{|x_1x_2| \prod_{i=1}^k|y_i|^2 : (x,y) \in K\right\}^2 \disc s\\
									&=& \sup_{c \in (0,2]} c^2 \cdot \sup \left\{\prod_{i=1}^k |y_i|^2 : \sum_i |y_i|^2=2-c\right\}^2\disc s\\
									&=& \left[\sup_{c \in (0,2]} c(2-c)^{k} \sup \left\{\prod_{i=1}^k |y_i|^2 : \sum_i |y_i|^2=1\right\}\right]^2\disc s\\
									&=&\left[\frac{2}{1+k}\left(\frac{2}{k}\right)^k\left(1-\frac{1}{1+k}\right)^k\right]^2\disc s\\
									&\leq & \disc s							
\end{eqnarray}
In line (5) we have used that
 \[1/k^k=\max\left\{\prod_{i=1}^k y_i^2 \mid (y_1,\dots,y_k) \in \RR^k,  \sum_{i=1}^{k} y_i^2=1\right\}.\]
Here the maxima lie at $|y_i|=1/\sqrt{k}$, $i \in \{1,\dots,k\}$.
For the case of $k=0$ one obtains $4\disc s$.
\end{proof}



\begin{proof}[Proof of Theorem \ref{thm:main2}]
%
	For any natural number $n$ let $s_n(x)$ denote the minimal polynomial of $\lambda^n$, and set $L_{K3}=3U\oplus 2 E_8$.
	By Lemmas \ref{lem:iso_QQ} and \ref{lem:iso_ZZ}, for some $n \in \NN$ we get an isometry
	$f\in O(L_{K3})$ with characteristic polynomial $s_n(x)(x-1)^{22-d}$  and such that $S=\ker s_n(f)$ is hyperbolic. 
    
	What remains is to modify $S=\ker s_n(f)$ in order to increase the determinant of $S$ such that $f|S$ preserves a chamber by Proposition \ref{prop:preserves_chamber}.
	Set $R=S^\perp=\ker(f-id)$.
	The primitive extension $S\oplus R\hookrightarrow L_{K3}$, provides us with an isomorphism of discriminant quadratic forms $q_S \cong q_R(-1)$. 
	By Chebotarev's density theorem there are infinitely many prime ideals $\mathfrak{p}<\O_k$ of degree one, split in $K/k$ such that 
	\[p \equiv 1 \mod (8\det R)\] 
	where $(p)=\ZZ\cap \mathfrak{p}$. 
	Choose one such $\mathfrak{p}$ with $p>\disc s_n(x)$.
	We can find $l\in \NN$ and $t \in \O_k$ with $\mathfrak{p}^l=t\O_k$, and then 
	\[|\det S(t^{2})|=|\det S|p^{2l}>\disc s_n(x).\]
	For primes $p'\neq p$, $S(t^{2})\otimes \ZZ_{p'} \rightarrow S\otimes \ZZ_{p'}$, $x\mapsto tx$ is an isometry. Hence, 
	\[(q_{S(t^{2})})_{p'}\cong (q_{S})_{p'} \cong (q_{R})_{p'}(-1)\cong (q_{R(p^{l})})_{p'}(-1).\]
	By Lemma \ref{lem:twist_split}	
	\[(q_{S(t^{2})})_{p}\cong p^{-2l} \left(\begin{matrix} 0 & 1 \\ 1 & 0 \end{matrix}\right)\cong p^{-2l}\left(\begin{matrix} 1 & 0 \\ 0 & -1 \end{matrix}\right).\]
	Further, a direct computation reveals 
	\[(q_{R(p^{2l})})_{p}(-1)\cong p^{-2l} \left(\begin{matrix} 1 & 0 \\ 0 & \det R \end{matrix}\right).\]
	The two forms are isomorphic if and only if both $-1$ and $\det R$ are of the same square class in $\ZZ_p^\times/\left(\ZZ_p^\times\right)^2$.
	This is computed by the Legendre symbols $\left(\frac{-1}{p}\right)$ and $\left(\frac{\det R}{p}\right)$.
	Since $p\equiv 1 \mod (8\det R)$, we have $\left(\frac{-1}{p}\right)=1$ and $\left(\frac{\det R}{p}\right)=\left(\frac{p}{\det R}\right)=1$.
	Piecing together the isometries for the different primes, we have constructed a glue map $q_{S(t^2)}\cong q_{R(p^{2t})}(-1)$. Its graph provides us with an overlattice of $S(t^2)\oplus R(p^{2t})$ isomorphic to $L_{K3}$.
	In other words, we have a primitive embedding of $S(t^{2})$ into $L_{K3}$ with orthogonal complement $R(p^{2l})$. By construction, $f|S(t^{2})$ preserves a chamber, and after replacing $f$ by a sufficiently divisible power $f^k$, we may glue it to the identity on $R(p^{2t})$ to obtain an isometry $f^k|_{S(t^{2})} \oplus id_{R(p^{2t})}$ which extends to $L$. We can apply Lemma \ref{lem:torelli_k3} (2) to get a projective K3 surface $X$ and $F \in \Aut(X)$ with dynamical degree some power of $\lambda$. 
\end{proof}

\subsection{Proof of Lemma \ref{lem:iso_QQ}}
Except for the signature condition, Lemma \ref{lem:iso_QQ} is a special case of \cite[Cor. 9.2, Prop. 11.9]{bayer-fluckiger:rationa_isometries}. In what follows, we inspect the original proof. For the readers convenience, we recall some of the notation involved.
We need only the following special case: $k=\QQ$ and $\Sigma_k$ is the set of its places.
Then $q$ is the quadratic form on $L\otimes \QQ$. An isometry
$t\in O(q)$ induces the structure of a self-dual torsion $\QQ[x]$-module on $L_\QQ$ via $p(x).v=p(t)v$ for $v\in L_\QQ$ and $p(x) \in \QQ[x]$.
We set 
\[M_0=\left(\QQ[x]/(x-1)\right)^{22-r}, \quad M_1=\QQ[x]/s(x), \quad \mbox{ and } \quad M=M_0\oplus M_1.\]
Then $\mathcal{C}_{M,q}$ is the set of all collections of forms $C=\{q_i^\nu \}$ for $i \in I_0=\{0,1\}$ and $\nu \in \Sigma_k$, such that $q_i^\nu$ has an isometry with module $M_i\otimes \QQ_\nu$ and $q_0^\nu \oplus q_1^\nu$ is isomorphic to the localization $q^\nu=q\otimes \QQ_\nu$ of $q$ at $\nu$. 
For $i\in I_0$, we set 
\[T_i(C)=\{ \nu \in \Sigma_k \mid w(q_i^\nu)=1 \}\]
where $w(q_i^\nu)$ is the Hasse invariant of $q_i^\nu$. 
Let $\mathcal{F}_{M,q}$ be the subset of $\mathcal{C}_{M,q}$ such that for all $i \in I_0$, $T_i(C)$ is a finite set.

 The main step involved is
\begin{theorem}\cite[Thm. 10.8]{bayer-fluckiger:rationa_isometries}\label{thm:existence_rat_module}
	Let $M$ be a self-dual torsion $k[x]$-module which is finite dimensional as a $k$-vector space. Suppose that the quadratic form $q$ over the global field $k$ has an isometry with module $M$ over $k_\nu$ for all places $\nu$ of $k$. Then $q$ has an isometry with module $M$ if and only if there exists a collection $C=\{q_i^\nu \} \in \mathcal{F}_{M,q}$ such that for all $i \in I_0$, the cardinality of $T_i(C)$ is even. In this case $q_i^\nu=q_i\otimes \kappa_\nu$.
\end{theorem}

\begin{proof}[Proof of Lemma \ref{lem:iso_QQ}]
	By \cite[Prop. 11.9]{bayer-fluckiger:rationa_isometries}, we find $f\in O(q)$ with module $M$. Hence its characteristic polynomial has the desired form
	\[\det(xId-f)=s(x)(x-1)^{22-r}.\]
	If $q_1=q|S_\QQ$ has signature $(1,r-1)$, we are done, else $q_1$ has signature $(3,r-3)$. 
	By Theorem \ref{thm:existence_rat_module}, this provides us with a collection $C=\{q_i^\nu\} \in \mathcal{F}_{M,q}$ with
	$q_i^\nu=q_i\otimes \kappa_\nu$ such that $|T_i(C)|$ is even for $i\in \{0,1\}$.
	We set \[d_1\equiv\det q_1 \equiv s(1)s(-1) \mod \QQ^{\times 2}\]
	and
	\[d_0\equiv\det q_0\equiv d_1 \det q =-s(1)s(-1) \mod \QQ^{\times2}.\] Denote by $\Omega(M_i,d_i)$ the set of finite places $\nu$ of $\QQ$ such that for any $\epsilon \in \{0,1\}$ there is a quadratic space $Q$ over $\QQ_\nu$ with determinant $d_i$, Hasse invariant $w(Q)=\epsilon$ and which has an isometry with module $M_i$.
	By \cite[Prop. 11.9]{bayer-fluckiger:rationa_isometries} every finite place is in $\Omega( M_0,d_0)$. Hence 
	\[\Omega_{0,1}=\Omega(M_0,d_0) \cap \Omega(M_1,d_1)=\Omega(M_1,d_1),\] 
	which is non-empty by \cite[Lem. 9.4, 9.6]{bayer-fluckiger:rationa_isometries} and Chebotarev's density theorem for the degree two extension $\QQ[\lambda]$ of $\QQ[\lambda+\lambda^{-1}]$. 
	Choose a place $p\in \Omega_{0,1}$.
	We can define a new collection $\tilde{C}=\{\tilde{q}_i^\nu\}$ by
	$\tilde{q}_i^\nu=q_i^\nu$ for $i\in \{0,1\}$ and $\nu \neq p, \infty$, 
	For $\tilde{q}_i^\infty$, we take forms of signature $(1,d-1)$, respectively $(2,\rk L-2-d)$.
	At $p$ we just switch the Hasse invariants of $q_i^p$. By \cite[Lem. 9.6]{bayer-fluckiger:rationa_isometries}, $\tilde{C} \in \mathcal{C}_{M,q}$, and moreover $\tilde{C}\in \mathcal{F}_{M,q}$. Since the Hasse invariants have changed at two places, $|T_i(\tilde{C})|$ is still even.
	Finally, $\tilde{C}$ meets the conditions of Theorem \ref{thm:existence_rat_module} and the claim follows.
\end{proof}

\bibliographystyle{JHEPsort}
\bibliography{literature}{}

\end{document}